\documentclass[12pt,reqno]{amsart}
\usepackage[english]{babel}
\usepackage{mathrsfs, csquotes}
\usepackage{amsfonts,amssymb,amsmath,mathtools,commath,braket}

\usepackage{ucs} 
\usepackage[T1]{fontenc}
\usepackage[utf8]{inputenc}

\usepackage{color} \definecolor{bleu_sombre}{rgb}{0,0,0.6}  \definecolor{rouge_sombre}{rgb}{0.8,0,0}\definecolor{vert_sombre}{rgb}{0,0.6,0}
\usepackage[plainpages=false,colorlinks,linkcolor=bleu_sombre,
citecolor=rouge_sombre,urlcolor=vert_sombre,breaklinks]{hyperref}
\usepackage{enumerate}

\theoremstyle{plain}
\newtheorem{theorem}{Theorem}[section]
\newtheorem{lemma}[theorem]{Lemma}

\newtheorem{proposition}[theorem]{Proposition}

\theoremstyle{definition}
\newtheorem{remark}[theorem]{Remark}
\newtheorem*{notation}{Notation}

\newtheorem{definition}[theorem]{Definition}
\newtheorem{assumption}[theorem]{Assumption}

\newcommand{\R}{\mathbb{R}}
\newcommand{\N}{\mathbb{N}}

\renewcommand{\leq}{\leqslant}	
\renewcommand{\geq}{\geqslant}

\numberwithin{equation}{section}

\newcommand{\dd}{\mathrm{d}}

\begin{document}

\title[On the stability of the Schwartz class]{On the stability of the
  Schwartz class \\ under the magnetic Schr\"odinger flow}
\author[G. Boil]{G. Boil} \email{gregory.boil@univ-rennes1.fr}
\author[N. Raymond]{N. Raymond} \email{nicolas.raymond@univ-rennes1.fr
} \author[S. V\~u Ng\d{o}c]{S. V\~u Ng\d{o}c}
\email{san.vu-ngoc@univ-rennes1.fr}

\address[]{IRMAR, Universit\'e de Rennes 1, Campus de Beaulieu,
  F-35042 Rennes cedex, France}

\begin{abstract}
  We prove that the Schwartz class is stable under the magnetic
  Schr\"odinger flow when the magnetic $2$-form is non-degenerate and
  does not oscillate too much at infinity.
\end{abstract}

\maketitle

\section{Introduction}

\subsection{Motivation and context}
This paper is devoted to describing the solutions to the magnetic
Schr\"odinger equation. Let $\mathbf{B}$ be a smooth and closed
$2$-form on $\R^d$. Let $\mathbf{A} : \R^d\to\R^d$ be a $1$-form
(identified with a vector field) such that
$\dd \mathbf{A}=\mathbf{B}$. The magnetic Schr\"odinger operator is
the essentially self-adjoint differential operator
\[\mathscr{L}_{h}=(-ih\nabla-\mathbf{A})^2=\sum_{j=1}^d L^2_{j}\,,\]
where $h>0$ and, for all $j\in\{1,\ldots, j\}$,
$L_{j}=-ih\partial_{j}-A_{j}$. Its domain is given by
\[
\begin{split}
  \mathrm{Dom}(\mathscr{L}_{h}) &= \{\psi\in L^2(\R^d) :
  (-ih\nabla-\mathbf{A})\psi\in L^2(\R^d)\,,
  (-ih\nabla-\mathbf{A})^2\psi\in L^2(\R^d)\}\\
  &=\{\psi\in L^2(\R^d) : (-ih\nabla-\mathbf{A})^2\psi\in
  L^2(\R^d)\}\,.
\end{split}
\]
The time dependent magnetic Schr\"odinger equation is given by
\begin{equation}\label{eq.SE}
  -ih\partial_{t}\psi=\mathscr{L}_{h}\psi\,,\qquad 
  \psi(0)=\psi_{0}\in\mathrm{Dom}(\mathscr{L}_{h})\,.
\end{equation}
By Stone's theorem, this Cauchy problem admits a unique solution,
evolving in the domain of $\mathscr{L}_{h}$, and it is given by
\[
\forall t\in\R\,,\quad
\psi(t)=e^{\frac{it\mathscr{L}_{h}}{h}}\psi_{0}\,.
\]
By unitarity of the flow, we have
\[
\forall t\in\R\,,\quad \|\psi(t)\|=\|\psi_{0}\|\,, \quad
\|\mathscr{L}_{h}\psi(t)\|=\|\mathscr{L}_{h}\psi_{0}\|\,,
\]
where $\|\cdot\|$ denotes the usual norm on $L^2(\R^d)$. This norm
controls the rough phase space localization of the quantum state
$\psi(t)$; a natural question is to know to which extent a strong
phase space localization of $\psi_0$ is preserved by the flow. More
precisely, this paper was inspired by the following rather naive
question. Is it true that
\begin{equation}\label{eq.question}
  \psi_{0}\in\mathscr{S}(\R^d)\quad\Longrightarrow\quad \forall t\in\R\,,\quad\psi(t)\in\mathscr{S}(\R^d) \quad?
\end{equation}
If so, what kind of explicit control do we have in terms of the
Schwartz semi-norms?

These questions are motivated by the recent investigation of the
propagation of coherent states by the magnetic Hamiltonian flow in two
dimensions (see the Ph.~D. thesis of the first
author~\cite{theseBoil}). The present paper gives a positive answer to
\eqref{eq.question}. Our explicit estimates of the Schwartz semi-norms
(in terms of the semiclassical parameter $h$), combined with the use
of the Birkhoff normal form from \cite{RVN15}, turn out to be the key
ingredients in the study by~\cite{BVN18} of the propagation of
coherent states up to times of order $h^{-N}$, for all
$N\in\mathbb{N}$. This gives a quantum analog to the low energy (say
of order $\varepsilon$) classical propagation for times of order
$\varepsilon^{-\infty}$ (see \cite[Theorem 1.2]{RVN15}). Taking into
account the analysis of \cite{HKRVN16}, one can even hope to extend
these results to three dimensions where the classical dynamics has a
more complex behavior.

Independently of this motivation, the answer to \eqref{eq.question}
has an interest of its own, especially because it lives at the
confluence of two closely related domains: hypoellipticity and
semiclassical analysis with magnetic fields. On these vast subjects,
the literature is enormous, and we only refer to \cite{Hormander67,
  HN, HM88, HN92, T02, T04, FH10, R17}. In this paper, we will use
many classical ideas from these two contexts, and provide an
elementary and self-contained presentation.

\subsection{Main results}
Let us now describe our assumptions and results.

Let $\mathscr{P}$ be the class defined by
\[\mathscr{P}=\{\psi\in\mathscr{C}^\infty(\R^d) :
\forall\alpha\in\mathbb{N}^d\,,\exists
(C,m)\in\R_{+}\times\mathbb{R}_{+}\,,\forall x\in\R^d\,,
|\partial^\alpha\psi|\leq C\langle x\rangle^{m} \}\,.\]

The following assumption will hold throughout the paper, where we
identify $\mathbf{B}$ with its antisymmetric matrix obtained in the
usual basis $(\dd x_j\wedge \dd x_k; j<k)$.
\begin{assumption}\label{hyp.0}
  We assume that
  \begin{enumerate}[\rm i.]
  \item $\mathbf{A}$ belongs to $\mathscr{P}$ (in particular $\mathbf{B}\in\mathscr{P}$),
  \item there exists $b_{0}>0$ such that, for all $x\in\R^d$,
    \[\mathrm{Tr}^+\mathbf{B}(x)\geq b_{0}\,,\]
    where $\mathrm{Tr}^+\mathbf{B}(x)$ denotes the sum of the moduli
    of the eigenvalues with positive imaginary part of the matrix
    $\mathbf{B}(x)$,
  \item for all $\alpha\in\mathbb{N}^d$, there exists $C>0$ such that,
    for all $x\in\R^d$,
    $\|\partial^\alpha\mathbf{B}(x)\|\leq C\|\mathbf{B}(x)\|$, where
    $\|\cdot\|$ denotes a norm on the space of matrices.
  \end{enumerate}
\end{assumption}
Assumption \ref{hyp.0} is stronger than really necessary as we can see
in our proofs. In this context, we will use the following lemma (see
\cite[Theorem 2.2]{HM96}).
\begin{lemma}\label{lem.minmag}
  We have
  \[\inf\mathrm{sp}(\mathscr{L}_{h})=h\inf_{x\in\R^d}\mathrm{Tr}^+\mathbf{B}(x)+o(h)\,.\]
  In particular, there exist $C>0$ and $h_{0}>0$, such that, for all
  $h\in(0,h_{0})$, $\mathscr{L}_{h}$ is invertible and
  \[\|\mathscr{L}_{h}^{-1}\|\leq Ch^{-1}\,.\]
\end{lemma}
In the following, we will always assume that $h$ is small enough and
such that $\mathscr{L}_{h}$ is invertible.

\begin{definition}
  For all $n\in\mathbb{N}$, we let
  \[\mathrm{Dom}(\mathscr{L}^n_{h})=\{\psi\in L^2(\R^d) :
  \forall\ell\in\{1,\ldots, n\} :
  \mathscr{L}_{h}^{\ell-1}\psi\in\mathrm{Dom}(\mathscr{L}_{h})\}\,.\]
  The operator $\mathscr{L}^n_{h}$ is defined by induction by
  \[\forall \psi\in\mathrm{Dom}(\mathscr{L}^n_{h})\,,\quad
  \mathscr{L}^n_{h}\psi=\mathscr{L}_{h}(\mathscr{L}^{n-1}_{h}\psi)\,.\]
\end{definition}
The operator $(\mathrm{Dom}(\mathscr{L}^n_{h}),\mathscr{L}^n_{h})$ is
self-adjoint and invertible. The following theorem proves some
magnetic elliptic estimates, showing that iterations of the magnetic
laplacian $\mathscr{L}_h$ control iterations of the magnetic
derivatives $(L_j)_{1\leq j \leq d}$. This will be an important tool
on the proof of the main result of the paper.
\begin{theorem}\label{thm.elliptic}
  Let Assumption~\ref{hyp.0} hold. Let $n\in\mathbb{N}$. There
  exist $h_{0}>0$ and $C>0$ such that, for all $h\in(0,h_{0})$, and
  all $\psi\in\mathrm{Dom}(\mathscr{L}^n_{h})$,
  \begin{equation}\label{eq.elliptic}
    \sum_{\sigma\in\mathfrak{A}(2n)} \|L_{\sigma}\psi\|\leq Ch^{-3n/2}\|\mathscr{L}_{h}^n\psi\|\,,
  \end{equation}
  where, for $k\in\N$,
  $\mathfrak{A}(k)=\cup_{p=0} ^k \{1,\ldots,d\}^{\{1,\ldots,p\}}$, and
  for $p\in\N$, for $\sigma\in\{1,\ldots,d\}^{\{1,\ldots,p\}}$,
  $L_{\sigma} := L_{\sigma(1)}\ldots L_{\sigma(p)}$, with the convention $L_{\emptyset}=\mathrm{Id}$.
\end{theorem}
In the case where $\mathbf{A}$ is
bounded, Theorem \ref{eq.elliptic} is closely related to \cite[Theorem
3]{T02}, which deals with the context of general G\aa rding
inequalities.

\begin{definition}
  For all $k\in\mathbb{N}$ and all $\psi\in\mathscr{S}(\R^d)$, we let
  \[p_{k}(\psi)=\underset{\underset{ |\alpha|+|\beta|\leq
      k}{(\alpha,\beta)\in\mathbb{N}^{2d}}}{\max}\|x^\alpha \partial^\beta
  \psi\|_{\infty}\,.\]
\end{definition}

We can now state the main result of this paper.
\begin{theorem}\label{thm.evol}
  Let Assumption~\ref{hyp.0} hold.  For all $t\in\R$, we have
  \[e^{\frac{it\mathscr{L}_{h}}{h}}\mathscr{S}(\R^d)\subset
  \mathscr{S}(\R^d)\,.\]
  More precisely, for all $M\in\mathbb{N}^*$, for all
  $k\in\mathbb{N}$, there exist $h_{0}>0$, $C>0$, $N\in\mathbb{N}^*$
  and $K\in\mathbb{N}$, such that, for all $h\in(0,h_{0})$, and for
  all $\psi_{0}\in \mathscr{S}(\R^d)$, and all $t\in[0,h^{-M}]$,
  \[p_{k}(e^{\frac{it\mathscr{L}_{h}}{h}}\psi_{0})\leq
  Ch^{-N}p_{K}(\psi_{0})\,.\]
\end{theorem}
Theorem \ref{thm.evol} is related to the (pseudo-differential)
analysis in \cite[Section 7]{T04}. In this work, under the assumption
that the derivatives of order two or higher of the symbol of the
propagator should be bounded, a parametrix of the evolution operator
was constructed. Closely related is also the paper~\cite[Corollary
2.11]{Rob07}, based on the analysis of coherent states, where the
derivatives of order three or higher of the symbol have to be
bounded. Our approach here is more directly related to the structure
of the magnetic Laplacian, and is reminiscent of the analysis
in~\cite{Hormander67} of the ellipticity of certain algebras of
non-commuting vector fields.

\subsection{Organisation of the proofs}
In Section \ref{sec.elliptic}, we prove Theorem \ref{thm.elliptic} by
using a regularization argument involving exponentially weighted
estimates and commutator estimates. In Section \ref{sec.evol}, we
apply our magnetic elliptic estimates to prove Theorem \ref{thm.evol}.

\newpage
\section{Magnetic elliptic estimates}\label{sec.elliptic}
This section is devoted to the proof of Theorem \ref{thm.elliptic}.

\subsection{Density argument}\label{sec.density}
Let us explain here why it is sufficient to prove Theorem
\ref{thm.elliptic} for $\psi\in\mathscr{S}(\R^d)$.

Let us consider $f\in L^2(\R^d)$. There is a unique
$\psi\in\mathrm{Dom}(\mathscr{L}^n_{h})$ such that
$\mathscr{L}^n_{h}\psi=f$. Consider
$f_{k}\in\mathscr{C}^\infty_{0}(\R^d)$ converging to $f$ in
$L^2(\R^d)$ and consider the unique
$\psi_{k}\in\mathrm{Dom}(\mathscr{L}^n_{h})$ such that
$\mathscr{L}^n_{h}\psi_{k}=f_{k}$. Note that, by continuity of
$(\mathscr{L}^n_{h})^{-1}$, $\psi_{k}$ converges to $\psi$ in
$L^2(\R^d)$.

\begin{lemma}\label{lem.agmon}
  We let
  \[H^\infty_{\exp}(\R^d)=\{\psi\in L^2(\R^d) : \forall
  \alpha\in\mathbb{N}^d\,,\exists \beta>0 : e^{\beta\langle
    x\rangle}\partial^\alpha\psi\in L^2(\R^d)\}\,.\]
  Consider $f\in H_{\exp}^\infty(\R^d)$. Then, the unique solution
  $u\in\mathrm{Dom}(\mathscr{L}_{h})$ to $\mathscr{L}_{h}u=f$
  satisfies $u\in H^\infty_{\exp}(\R^d)$.
\end{lemma}

\begin{proof}
  The proof follows from the classical Agmon estimates (see
  \cite{Agmon85, H88}). Consider $\beta>0$ such that
  $e^{\beta\langle x\rangle}f\in L^2(\R^d)$. Let $\varepsilon>0$ and
  $\Phi_\varepsilon=\beta\min(\langle x\rangle,\varepsilon^{-1})$.  We
  have the Agmon formula (see \cite[Section 4.2]{R17}):
  \[\langle\mathscr{L}_{h}u,
  e^{2\Phi_\varepsilon}u\rangle=\int_{\R^d}|(-ih\nabla-\mathbf{A})e^{\Phi_\varepsilon}u|^2
  \dd x-h^2\|\nabla\Phi_\varepsilon e^{\Phi_\varepsilon}u\|^2\,.\]
  In particular,
  \begin{equation}\label{eq.agmonineg}
    \int_{\R^d}|(-ih\nabla-\mathbf{A})e^{\Phi_\varepsilon}u|^2\dd x-\beta^2h^2\|e^{\Phi_\varepsilon}u\|^2\leq \|e^{\Phi_\varepsilon}f\|\|e^{\Phi_\varepsilon}u\|\,.
  \end{equation}
  On the other hand, by Lemma \ref{lem.minmag}, we get, for some
  $c>0$,
  \[(hc-\beta^2h^2)\|e^{\Phi_\varepsilon}u\|^2\leq\int_{\R^d}|(-ih\nabla-\mathbf{A})e^{\Phi_\varepsilon}u|^2\dd
  x-\beta^2h^2\|e^{\Phi_\varepsilon}u\|^2\,.\]
  Choosing $\beta$ small enough, we get, for some $C(h)>0$ independent
  of $\varepsilon$,
  \[\|e^{\Phi_\varepsilon}u\|\leq C(h)\|e^{\Phi_\varepsilon}f\|\,.\]
  Then, we take the limit $\varepsilon\to 0$ and apply Fatou's Lemma
  to find
  \begin{equation}\label{eq.AgmonL2}
  \|e^{\beta\langle x\rangle}u\|\leq C(h)\|e^{\beta\langle
    x\rangle}f\|\,.
 \end{equation} 
  Coming back to \eqref{eq.agmonineg}, and replacing $\beta$ by
  $\tilde\beta<\beta$, we get
  \[\int_{\R^d}|e^{\Phi_\varepsilon}(-ih\nabla-\mathbf{A}-ih\nabla\Phi_{\varepsilon})u|^2\dd x-\tilde\beta^2h^2\|e^{\Phi_\varepsilon}u\|^2\leq \|e^{\Phi_\varepsilon}f\|\|e^{\Phi_\varepsilon}u\|
\,.\]
Thus, 
  \[\begin{split}  
  \frac{h^2}{2}\int_{\R^d}|e^{\Phi_\varepsilon}\nabla u|^2\dd x &\leq \|e^{\Phi_\varepsilon}f\|\|e^{\Phi_\varepsilon}u\|+\tilde\beta^2h^2\|e^{\Phi_\varepsilon}u\|^2+2\|(ih\nabla\Phi_{\varepsilon}+\mathbf{A})e^{\Phi_{\varepsilon}}u\|^2\\
&\leq  \|e^{\Phi_\varepsilon}f\|\|e^{\Phi_\varepsilon}u\|+\tilde\beta^2h^2\|e^{\Phi_\varepsilon}u\|^2+4\|\nabla\Phi_{\varepsilon}e^{\Phi_{\varepsilon}}u\|^2+4\|\mathbf{A}e^{\Phi_{\varepsilon}}u\|^2\,.
\end{split}
 \] 
 Using that $\mathbf{A}\in\mathscr{P}$, \eqref{eq.AgmonL2}, and Fatou's Lemma, we get $e^{\tilde\beta\langle x\rangle}\nabla u\in L^2(\R^d)$.

Considering the equation $\mathscr{L}_{h}u=f$, we get, in the sense of tempered distributions,
\begin{equation}\label{eq.Deltau}
-h^2\Delta u=f-|\mathbf{A}|^2u-ih(\nabla\cdot\mathbf{A})u-2ih(\mathbf{A}\cdot\nabla) u\,.
\end{equation}
  
Noticing that we have just controlled the terms of order at most one, we deduce that, for
  some $\beta>0$,
  \[e^{\beta\langle x\rangle}\Delta u\in L^2(\R^d)\,,\]
and also
  \[
  \Delta \left(e^{\beta\langle x\rangle} u\right)\in L^2(\R^d)\,.
  \]
  By Fourier transform, we get
  \[e^{\beta\langle x\rangle}u\in H^2(\R^d)\,,\]
  which implies that, for all $\alpha\in\mathbb{N}^d$ with
  $|\alpha|\leq 2$,
  \[e^{\beta\langle x\rangle}\partial^\alpha u\in L^2(\R^d)\,.\]
  The higher order derivatives can be controlled by induction (taking
  successive derivatives of \eqref{eq.Deltau}).
\end{proof}

\begin{remark}
  By the Sobolev embeddings, we have
  $H^\infty_{\exp}(\R^d)\subset\mathscr{S}(\R^d)$.
\end{remark}

By Lemma \ref{lem.agmon}, we see that
$\psi_{k}\in\mathscr{S}(\R^d)$. Assume that \eqref{eq.elliptic} holds
for any $\psi \in \mathscr{S}(\R^d)$, with
$\psi=\psi_{k}-\psi_{\ell}$, for all $(k,\ell)\in\mathbb{N}^2$, this
shows that, for all $\sigma\in\mathfrak{A}(2n)$,
$(L_{\sigma}\psi_{n})_{n\in\mathbb{N}}$ is a Cauchy sequence in
$L^2(\R^d)$. Thus, in the sense of distributions,
$L_{\sigma}\psi\in L^2(\R^d)$. It remains to use again
\eqref{eq.elliptic} with $\psi=\psi_{k}$ and to take the limit
$k\to+\infty$.

\subsection{Preliminary lemmas}

\begin{lemma}\label{lem.prelim}
  For all $\psi\in\mathrm{Dom}(\mathscr{L}_{h})$, we have
  \[\|(-ih\nabla-\mathbf{A})\psi\|^2\leq\frac{1}{2}\|\psi\|^2_{\mathscr{L}_{h}}\]
  where
  $\|\psi\|^2_{\mathscr{L}_{h}}= \|\psi\|^2+ \|\mathscr{L}_h
  \psi\|^2$.
\end{lemma}
\begin{proof}
  We recall that, by definition of the domain, for all
  $\psi\in\mathrm{Dom}(\mathscr{L}_{h})$,
  \[\langle\mathscr{L}_{h}\psi,\psi\rangle=\|(-ih\nabla-\mathbf{A})\psi\|^2=\sum_{j=1}^d\|L_{j}\psi\|^2\,,\]
  so that
  \[\sum_{j=1}^d\|L_{j}\psi\|^2\leq
  \|\psi\|\|\mathscr{L}_{h}\psi\|\,.\]
\end{proof}

\begin{lemma}\label{lem.ell-B}
  There exist $C>0$ and $h_{0}>0$ such that, for all
  $\psi\in\mathscr{S}(\R^d)$ and all $h\in(0,h_{0})$,
  \begin{equation}\label{eq.B}
    \|\langle\mathbf{B}\rangle\psi\|+\|\langle\mathbf{B}\rangle^{\frac{1}{2}}(-ih\nabla-\mathbf{A})\psi\|\leq Ch^{-1}\|\psi\|_{\mathscr{L}_{h}}\,.
  \end{equation}
  Moreover, for all $\psi\in\mathrm{Dom}(\mathscr{L}_{h})$, we have
  \[\mathbf{B}\psi\in L^2(\R^d)\quad \mbox{ and }\quad
  |\mathbf{B}|^{\frac{1}{2}}(-ih\nabla-\mathbf{A})\psi\in
  L^2(\R^d)\,,\]
  and \eqref{eq.B} holds for $\psi\in\mathrm{Dom}(\mathscr{L}_{h})$.

\end{lemma}
\begin{proof}
  By integration by parts and using Assumption \ref{hyp.0},
  \begin{equation}\label{eq.1}
    \begin{split}
      \int_{\R^d}\langle\mathbf{B}\rangle|(-ih\nabla-\mathbf{A})\psi|^2\dd x&=\langle \langle\mathbf{B}\rangle(-ih\nabla-\mathbf{A})\psi,(-ih\nabla-\mathbf{A})\psi\rangle\\
      &\leq\langle \mathscr{L}_{h}\psi,\langle\mathbf{B}\rangle\psi\rangle+C\|\langle\mathbf{B}\rangle\psi\|\|(-ih\nabla-\mathbf{A})\psi\|\\
      &\leq
      C\|\langle\mathbf{B}\rangle\psi\|\|\psi\|_{\mathscr{L}_{h}}\,.
    \end{split}
  \end{equation}
  Then, we have
  \[\int_{\R^d}|h\mathbf{B}|^2|\psi|^2\dd
  x=\sum_{(k,\ell)\in\{1,\ldots,d\}^2}
  \int_{\R^d}|hB_{k,\ell}|^2|\psi|^2\dd x\] and we write
  \begin{multline*}
    \sum_{(k,\ell)\in\{1,\ldots,d\}^2}\int_{\R^d}|hB_{k,\ell}|^2|\psi|^2\dd x=\sum_{(k,\ell)\in\{1,\ldots,d\}^2}|\langle[L_{k},L_{\ell}]\psi, hB_{k,\ell}\psi\rangle|\\
    \leq
    Ch\|\langle\mathbf{B}\rangle\psi\|\|(-ih\nabla-\mathbf{A})\psi\|+Ch\int_{\R^d}\langle\mathbf{B}\rangle|(-ih\nabla-\mathbf{A})\psi|^2\dd
    x\,,
  \end{multline*}
  where we used an integration by parts and Assumption \ref{hyp.0}.

  By \eqref{eq.1}, it follows
  \[\int_{\R^d}|\mathbf{B}|^2|\psi|^2\dd x\leq
  Ch^{-1}\|\langle\mathbf{B}\rangle\psi\|\|\psi\|_{\mathscr{L}_{h}}\]
  and then
  \[\|\langle\mathbf{B}\rangle\psi\|\leq
  Ch^{-1}\|\psi\|_{\mathscr{L}_{h}}\,.\]
  Using again \eqref{eq.1}, the conclusion follows.
\end{proof}

\subsection{Case $n=1$}
The estimate of Theorem \ref{thm.elliptic} is obvious when $n=0$. Let us consider the case when $n=1$ to explain the principle producing these estimates.

\begin{lemma}\label{lem.ordre2sch}
  There exist $C>0$, $h_{0}>0$ such that, for all
  $\psi\in\mathscr{S}(\R^d)$ and all $h\in(0,h_{0})$,
  \[\|L^2_{1}\psi\|+\|L^2_{2}\psi\|+\|L_{1}L_{2}\psi\|+\|L_{2}L_{1}\psi\|\leq
  Ch^{-1}\|\mathscr{L}_{h}\psi\|\,.\]
\end{lemma}

\begin{proof}
  Let us consider $\psi\in\mathscr{S}(\R^d)$ and let
  \[\mathscr{L}_{h}\psi=f\,.\]
  Consider $j\in\{1,\ldots, d\}$. We have
  \[\mathscr{L}_{h}(L_{j}\psi)=L_{j}f+[\mathscr{L}_{h},
  L_{j}]\psi\,,\] and
  \[\|(-ih\nabla-\mathbf{A})(L_{j}\psi)\|^2=\langle L_{j}f,
  L_{j}\psi\rangle+\langle[\mathscr{L}_{h}, L_{j}]\psi,
  L_{j}\psi\rangle\,,\] We have
  \[[\mathscr{L}_{h}, L_{j}]=\sum_{k=1}^d
  [L^2_{k},L_{j}]=\sum_{k=1}^d\left( [L_{k},L_{j}]L_{k}+L_{k}[L_{k},
    L_{j}]\right)\,.\] Thus,
  \[|\langle[\mathscr{L}_{h}, L_{j}]\psi, L_{j}\psi\rangle|\leq
  C\|\mathbf{B}\psi\|\|(-ih\nabla-\mathbf{A})\psi\|+C\int_{\R^d}
  |\mathbf{B}||(-ih\nabla-\mathbf{A})\psi| ^2\dd x\,.\]
  We have, for all $\varepsilon\in(0,1)$,
  \[|\langle L_{j}f, L_{j}\psi\rangle|\leq
  \varepsilon\|L_{j}^2\psi\|^2+C\|f\|^2\,,\] and then
  \begin{multline*}
    \|(-ih\nabla-\mathbf{A})(L_{j}\psi)\|^2\\
    \leq
    C\|f\|^2+C\|\mathbf{B}\psi\|\|(-ih\nabla-\mathbf{A})\psi\|+C\int_{\R^d}
    |\mathbf{B}||(-ih\nabla-\mathbf{A})\psi| ^2\dd x\,.
  \end{multline*}
  With Lemma \ref{lem.prelim}, noting that
  $\|\psi\|_{\mathscr{L}_h} ^2 \leq C (1+ h^{-2})\|f\|^2$, we find
  \[\|(-ih\nabla-\mathbf{A})(L_{j}\psi)\|^2\leq Ch^{-3}\|f\|^2\,.\]
\end{proof}

\subsection{Induction}
Let $n\in\mathbb{N}^*$. Let us assume that, for all
$k\in\{1,\ldots, n\}$, the ellipticity property \eqref{eq.elliptic} is
true.  Let us consider $f,\psi\in\mathscr{S}(\R^d)$ such that
\[\mathscr{L}_{h}\psi=f\,.\]
Consider $\sigma\in\mathfrak{A}(2n)$. Since the functions are in the
Schwartz class, all the following computations are justified.

We have
\[\mathscr{L}_{h}L_{\sigma}\psi=L_{\sigma}f+[\mathscr{L}_{h},
L_{\sigma}]\psi\,,\] and then
\begin{equation}\label{eq.Lalpha}
  \langle\mathscr{L}_{h}L_{\sigma}\psi,L_{\sigma}\psi\rangle=\langle L_{\sigma}f, L_{\sigma}\psi\rangle+\langle[\mathscr{L}_{h}, L_{\sigma}]\psi, L_{\sigma}\psi\rangle\,.
\end{equation}
By using the Cauchy-Schwarz inequality and the induction assumption,
we have for all $\varepsilon\in(0,1)$,
\[|\langle L_{\sigma}f, L_{\sigma}\psi\rangle|\leq
Ch^{-3n}\|\mathscr{L}^n_{h}f\|\|\mathscr{L}^n_{h}\psi\|\leq
Ch^{-3n}\|\mathscr{L}^{n+1}_{h}\psi\|^2+Ch^{-3n}\|\mathscr{L}^{n}_{h}\psi\|^2\,,\]
so that, by Lemma \ref{lem.minmag},
\[|\langle L_{\sigma}f, L_{\sigma}\psi\rangle|\leq
Ch^{-(3n+2)}\|\mathscr{L}^{n+1}_{h}\psi\|^2\,.\]
Let us now deal with
$\langle[\mathscr{L}_{h}, L_{\sigma}]\psi, L_{\sigma}\psi\rangle$. The
commutator $[\mathscr{L}_{h}, L_{\sigma}]$ is the sum of various
terms. Each of them is the composition of at most $2n-2$ of the
$L_{j}$ and with exactly one of the $B_{k,\ell}$. By commuting the
$B_{k,\ell}$ to put it on the left, and using Assumption \ref{hyp.0},
we get
\[\|[\mathscr{L}_{h}, L_{\sigma}]\psi\|\leq
C\sum_{\tau\in\mathfrak{A}(2n-2)}\|\langle\mathbf{B}\rangle
L_{\tau}\psi\|\,.\]
By applying Lemma \ref{lem.ell-B}, and then the induction assumption,
we get
\[\|[\mathscr{L}_{h}, L_{\sigma}]\psi\|\leq
Ch^{-1}\sum_{\tau\in\mathfrak{A}(2n-2)}\|
L_{\tau}\psi\|_{\mathscr{L}_{h}}\leq
Ch^{-(3n-1)/2}\|\mathscr{L}^n_{h}\psi\|\,.\] Thus, we deduce
\[|\langle\mathscr{L}_{h}L_{\sigma}\psi,L_{\sigma}\psi\rangle|\leq
Ch^{-3n-1/2}\|\mathscr{L}^{n+1}_{h}\psi\|^2\,.\]
This shows that, for all $\gamma\in\mathfrak{A}(2n+1)$,
\begin{equation}\label{eq.2n+1}
  \|L_{\gamma}\psi\|\leq Ch^{-(3n/2+1)}\|\mathscr{L}^{n+1}_{h}\psi\|\,.
\end{equation}
Now, we want to get the control for $\gamma \in
\mathfrak{A}(2n+2)$.
Let $\sigma \in \mathfrak{A}(2n+1)$. We consider again
\eqref{eq.Lalpha}. By integration by parts, we can write
\[\langle L_{\sigma}f, L_{\sigma}\psi\rangle=\langle
L_{\check\sigma}f, L_{\hat\sigma}\psi\rangle\,,\]
with $\check\sigma \in \mathfrak{A}(2n)$ and
$\hat\sigma\in\mathfrak{A}(2n+2)$. Thus, by Cauchy-Schwarz, and the
induction assumption, for all $\varepsilon>0$, there exists $C>0$ such
that
\begin{equation}\label{eq.RHS1}
  |\langle L_{\sigma}f, L_{\sigma}\psi\rangle|\leq\varepsilon\|L_{\hat\sigma}\psi\|^2+Ch^{-3n/2}\|\mathscr{L}^{n+1}_{h}\psi\|^2\,.
\end{equation}
As previously, we have
\[\|[\mathscr{L}_{h}, L_{\sigma}]\psi\|\leq
C\sum_{\tau\in\mathfrak{A}(2n-1)}\|\langle\mathbf{B}\rangle
L_{\tau}\psi\|\,.\]
We use Lemma \ref{lem.ell-B} and \eqref{eq.2n+1} to find
\[\|[\mathscr{L}_{h}, L_{\sigma}]\psi\|\leq
Ch^{-1}\sum_{\tau\in\mathfrak{A}(2n-1)}\|
L_{\tau}\psi\|_{\mathscr{L}_{h}}\leq
Ch^{-(3n/2+2)}\|\mathscr{L}^{n+1}_{h}\psi\|\,.\]
Then, with Cauchy-Schwarz and \eqref{eq.2n+1}, we get
\begin{equation}\label{eq.RHS2}
  |\langle[\mathscr{L}_{h}, L_{\sigma}]\psi, L_{\sigma}\psi\rangle|\leq Ch^{-(3n+3)}\|\mathscr{L}^{n+1}_{h}\psi\|^2\,.
\end{equation}
From \eqref{eq.Lalpha}, \eqref{eq.RHS1}, and \eqref{eq.RHS2}, summing
over $\sigma \in \mathfrak{A}(2n+1)$, and choosing $\varepsilon$ small
enough, we get \eqref{eq.elliptic} with $n$ replaced by $n+1$.

This achieves the proof of Theorem \ref{thm.elliptic} when
$\psi\in\mathscr{S}(\R^d)$ and it remains to use the discussion of
Section \ref{sec.density}.

\section{Application to the evolution problem}\label{sec.evol}
We can now prove Theorem \ref{thm.evol}. Let
$\psi_{0}\in\mathscr{S}(\R^d)$. We denote by $\psi(\cdot)$ the
solution to the Schr\"odinger equation \eqref{eq.SE}. 

\begin{notation}
For $\kappa, \lambda \in \N$, we define $\Pi_{\kappa,\lambda}$
the set of the operators $P$ that are composition of operators taken
among $(L_j)_{1\leq j\leq d}$ and and $(x_k)_{1\leq k\leq d}$ with
$\kappa$ occurences of $x$ and $\lambda$ occurences of $L$. We also set $\displaystyle{\Pi=\bigcup_{(\kappa,\lambda)\in\N^2}\Pi_{\kappa,\lambda}}$.
\end{notation}

The aim of this
section is to prove the following proposition.
\begin{proposition}\label{prop.P}
  Let $(\kappa,\lambda)\in\mathbb{N}^2$ and consider $P\in\Pi_{\kappa,\lambda}$ . There exist
  $h_{0}>0$, $C\geq 0$ and $N\in\mathbb{N}$ such that, for all
  $t\geq 0$, and all $\psi_{0}\in\mathscr{S}(\R^d)$, for
  $\lambda \leq 2n \leq \lambda+1$,
  \[\|P\psi(t)\|\leq C h^{-N}\left(1+ t^\kappa\right) \sum_{|\alpha|
    \leq \kappa : |\alpha| + \nu \leq \kappa + n} \|\mathscr{L}_h ^\nu
  x^\alpha \psi_0\|\,.\]
\end{proposition}
This proposition implies the control of the Schwartz semi-norms and
achieves the proof of Theorem \ref{thm.evol}. Indeed, from Sobolev embeddings, for all $k\in\N$, there exists $K \in \N$ such that for all $ f \in \mathscr{S}(\R^d)$, 
\[p_k(f) \leq \max_{|\alpha|, |\beta| \leq K}\ \|
x^\alpha \partial^\beta f\|\,.\]
Using now that $\partial_j = (-ih)^{-1} (L_j +A_j)$, and
$A \in \mathscr{P}$, there exists $C>0$ and $N\in\mathbb{N}$ such
that, for all $f\in L^2(\R^d)$, if $\norm{P f}<+\infty$ for all
$P\in \Pi$, then $x^\alpha\partial^\beta f\in L^2(\R^d)$ for all
$\alpha,\beta\in \N^d$, and
\[\|x^\alpha \partial^\beta f \| \leq Ch ^{-N} \sum_{P \in R} \|P f\|\,,\]
where $R$ is a finite part of $\Pi$. Then, Proposition~\ref{prop.P} implies, for $N$, $\kappa$ and $n$ large enough,
\[\|x^\alpha \partial^\beta \psi(t)\| \leq Ch^{-N} \sum_{P\in R} \left\|P \psi(t)\right\| \leq Ch^{-N} (1+t^\kappa) \sum_{|\gamma| \leq \kappa, |\nu|\leq n} \|\mathscr{L}_h ^\nu x^\gamma \psi_0\|\,.\]
Finally, we conclude by
\[\|\mathscr{L}_h ^\nu x^\alpha \psi_0\| \leq C \|(1+x^2)^m\mathscr{L}_h ^\nu x^\alpha \psi_0\|_\infty \leq C p_K(\psi_0)\,,\]
for $m$, $K$ large enough.
%

\subsection{Case when $\kappa=0$}
For all $t\geq 0$, we have, by definition,
\[\psi(t)=e^{\frac{it}{h}\mathscr{L}_{h}}\psi_{0}\,.\]
For all $\ell\in\mathbb{N}$, we get
\[\mathscr{L}_{h}^\ell\psi(t)=e^{\frac{it}{h}\mathscr{L}_{h}}\mathscr{L}^\ell_{h}\psi_{0}\,.\]
Applying Theorem \ref{thm.elliptic}, this establishes the estimate of
Proposition \ref{prop.P} when $\kappa=0$.

\subsection{Case when $\kappa=1$}
Before starting the induction procedure, let us understand first the
mechanism with only one occurence of $x$. Let $j\in\{1,\ldots,
d\}$. We have
\[-ih\partial_{t}(x_{j}\psi)=\mathscr{L}_{h}(x_{j}\psi)+[x_{j},\mathscr{L}_{h}]\psi\,.\]
Note that $[x_{j},\mathscr{L}_{h}]=2hL_{j}$. With the Duhamel formula,
we have, for all $t\geq 0$,
\begin{equation}\label{eq.duh}
  x_{j}\psi(t)=e^{\frac{it}{h}\mathscr{L}_{h}}(x_{j}\psi_{0})+\int_{0}^t e^{\frac{i(t-s)}{h}\mathscr{L}_{h}}2L_{j}\psi(s)\dd s\,.
\end{equation}
With Lemma \ref{lem.prelim}, we get
\[\|x_{j}\psi(t)\|\leq
\|x_{j}\psi_{0}\|+2\int_{0}^t\|L_{j}\psi(s)\|\dd s\leq
\|x_{j}\psi_{0}\|+\sqrt{2}\int_{0}^t\|\psi(s)\|_{\mathscr{L}_{h}}\dd s
\,.\] Since the evolution is unitary, we get, for all $t\geq 0$,
\[\|x_{j}\psi(t)\|\leq
\|x_{j}\psi_{0}\|+2\int_{0}^t\|L_{j}\psi(s)\|\dd s\leq
\|x_{j}\psi_{0}\|+t\sqrt{2}\|\psi_{0}\|_{\mathscr{L}_{h}} \,.\]
More generally, with \eqref{eq.duh}, we have, for all
$\ell\in\mathbb{N}$,
\[\mathscr{L}^\ell_{h}x_{j}\psi(t)=e^{\frac{it}{h}\mathscr{L}_{h}}\mathscr{L}^{\ell}_{h}(x_{j}\psi_{0})+\int_{0}^t
e^{\frac{i(t-s)}{h}\mathscr{L}_{h}}2\mathscr{L}^{\ell}_{h}L_{j}\psi(s)\dd
s\,.\] so that
\[\begin{split}
  \|\mathscr{L}^\ell_{h}x_{j}\psi(t)\|&\leq\|\mathscr{L}^{\ell}_{h}(x_{j}\psi_{0})\|+2\int_{0}^t \|\mathscr{L}^{\ell}_{h}L_{j}\psi(s)\|\dd s\\
  &\leq \|\mathscr{L}^{\ell}_{h}(x_{j}\psi_{0})\|+Ch^{-N}\int_{0}^t\|\mathscr{L}_{h}^{\ell+1}\psi(s)\|\dd s\\
  &\leq\|\mathscr{L}^{\ell}_{h}(x_{j}\psi_{0})\|+Cth^{-N}\|\mathscr{L}_{h}^{\ell+1}\psi_{0}\|\,.
\end{split}
\]
It remains to apply Theorem \ref{thm.elliptic} and to commute the
$x_{j}$ with the $L_{k}$.

\subsection{Induction}
Let us now end the proof of Proposition \ref{prop.P} by induction. We set the following two induction assumptions. For $\kappa\in\N$, let
\begin{multline*}
  \mathscr{Q}_\kappa : \forall n\in\N, \, \forall \alpha\in\N^d, \, |\alpha| =\kappa, \, \exists N \in \N \ \text{s.t.} \\
  \|\mathscr{L}_h ^n x^\alpha \psi(t)\| \leq \|\mathscr{L}_h ^n
  x^\alpha \psi_0\| +C \sum_{|\beta| \leq \kappa-1 : |\beta| + \nu
    \leq \kappa + n}\|\mathscr{L}_h ^\nu x^\beta \psi_0\| h^{-N}
  (1+t^\kappa)\,;
\end{multline*}
and
\begin{equation*}
  \mathscr{P}_\kappa : \forall P \in \Pi_{\kappa,\lambda}, \ \|P \psi(t)\| \leq C \sum_{|\alpha| \leq \kappa : |\alpha| + \nu \leq \kappa +n} \|\mathscr{L}_h ^\nu x^\alpha \psi_0\| h^{-N} (1+t^\kappa), \ \text{for} \, \lambda \leq 2n \leq \lambda +1\,.
\end{equation*}
We have proved propositions $\mathscr{P}_0$, $\mathscr{Q}_0$ and
$\mathscr{Q}_1$. We assume now that for a given $\kappa\in \N^*$, for
any $k\leq\kappa$, $\mathscr{Q}_k$ and $\mathscr{P}_k$ hold, and we
prove $\mathscr{P}_{\kappa+1}$ and $\mathscr{Q}_{\kappa +1}$. We begin
with $\mathscr{Q}_{\kappa+1}$. Let $\alpha\in\N^d$, with
$|\alpha| = \kappa +1$ and $n\in\N$. We have
\[-ih\partial_t \big(\mathscr{L}_h ^n x^\alpha \psi(t)\big) =
\mathscr{L}_h ^{n+1} x^\alpha \psi(t) + \mathscr{L}_h ^n [x^\alpha,
\mathscr{L}_h] \psi(t)\,.\] Then, noting that
\[[x^\alpha, \mathscr{L}_h] = h P_1\]
with $P_1$ a sum of elements in $\Pi_{\kappa,1}$, we get from the
Duhamel formula
\begin{equation*}
  \mathscr{L}_h ^n x^\alpha \psi(t) = e^{\frac{it}{h}\mathscr{L}_h} \mathscr{L}_h ^n x^\alpha \psi_0 + i \int_0 ^t e^{i\frac{t-s}{h} \mathscr{L}_h} \mathscr{L}_h ^n P_1 \psi(s) \mathrm{d}s\,.
\end{equation*}
As $\mathscr{L}_h ^n P_1$ is a sum of elements in
$\Pi_{\kappa,(2n+1)}$, we can apply $\mathscr{P}_\kappa$, and
integrating in time and using the unitariness of
$e^{i\frac{t}{h}\mathscr{L}_h}$, we find, for some integer $N \in \N$
\begin{equation*}
  \|\mathscr{L}_h ^n x^\alpha \psi(t)\| \leq \|\mathscr{L}_h ^n x^\alpha \psi_0\| + C h^{-N} \sum_{|\beta| \leq \kappa : |\beta| + \nu \leq \kappa + n+1} \|\mathscr{L}_h ^{\nu} x^\beta \psi_0\| (1+t^{\kappa+1})
\end{equation*}
that proves $\mathscr{Q}_{\kappa + 1}$. It remains to prove
$\mathscr{P}_{\kappa+1}$. We consider so some
$P \in \Pi_{\kappa+1,\lambda}$, for a given $\lambda \in \N$. Then,
because of the commutation relation of the $(x_k)_{1\leq k \leq d}$
with the $(L_j)_{1 \leq j \leq d}$, that is
$[x_k,L_j] = -ih\delta_{k,j}$, there is $\alpha \in \N^d$,
$|\alpha| = \kappa+1$ such that
\[P = P_2 + P_3 x^\alpha\]
with $P_2$ a sum of elements in $\Pi_{\kappa,\lambda-1}$ and
$P_3 \in \Pi_{0,\lambda}$. So
\[\|P \psi(t)\| \leq \|P_2 \psi(t)\| + \|P_3 \psi(t)\|\,.\]
Then, applying $\mathscr{P}_\kappa$, we get, for some integer $N\in\N$
\begin{equation}\label{eq.P2}
  \|P_2 \psi(t)\| \leq C h^{-N} (1+t^\kappa) \sum_{|\beta|\leq \kappa : |\beta|+ \nu \leq \kappa + n} \|\mathscr{L}_h ^\nu x^\beta \psi_0\|
\end{equation}
and applying Theorem~\ref{thm.elliptic} along with
$\mathscr{Q}_{\kappa+1}$, for some integer $N\in\N$, and for
$\lambda \leq 2n \leq \lambda+1$, we have
\begin{equation}\label{eq.P3}
  \begin{split}
    \|P_3 x^\alpha \psi(t)\| &\leq C h^{-n}\|\mathscr{L}^n x^\alpha \psi(t)\| \\
    &\leq C h^{-N} (1+t^{\kappa+1}) \left( \|\mathscr{L}^n x^\alpha
      \psi_0\| + \sum_{|\beta| \leq \kappa : |\beta| + \nu \leq \kappa
        + n +1} \|\mathscr{L}_h ^{\nu} x^\beta \psi_0\| \right).
  \end{split}
\end{equation}
Now, gathering \eqref{eq.P2} and \eqref{eq.P3}, we find, for some
integer $N$,
\[\|P \psi(t)\| \leq C h^{-N} (1+t^{\kappa+1}) \sum_{|\beta| \leq
  \kappa +1: |\beta| + \nu \leq \kappa+n+1} \|\mathscr{L}_h ^\nu
x^\beta \psi_0\|\]
that proves $\mathscr{P}_{\kappa+1}$ and achieves the proof of
Proposition~\ref{prop.P}.

\subsection*{Acknowledgments}
The authors are grateful to Karel Pravda-Starov for sharing useful
references.

\end{document}